\documentclass[reqno,11pt]{amsart}
\usepackage{graphicx}    % standard LaTeX graphics tool when including figure files
\usepackage{amsfonts}
\usepackage{amssymb}
\usepackage[latin1]{inputenc}

\newtheorem{thm}{Theorem}[section]
\newtheorem{cor}[thm]{Corollary}
\newtheorem{lemma}[thm]{Lemma}

\newtheorem{prop}[thm]{Proposition}
\theoremstyle{definition}
\newtheorem{defn}[thm]{Definition}
\theoremstyle{remark}

\newtheorem{remark}[thm]{Remark}

\newcommand{\pb}[1]{\left\{#1\right\}}

\newcommand{\pp}[2]{\frac{\partial#1}{\partial#2}}

\newcommand{\we}{\wedge}

\newcommand{\F}{{\bf F}}
\newcommand{\G}{{\bf G}}

\newcommand{\X}{\mathcal X}

\newcommand{\D}{\mathcal D}

\newcommand{\R}{\mathbf{R}}

\newcommand{\diff}{{\rm d }}

\newif\ifprivate
\privatefalse

 \numberwithin{equation}{section}

\def\???{\ifprivate {\bf {???}} \marginpar{{\Huge {\bf ?}}}\else \fi}
\numberwithin{equation}{section}

\title{Coupling symmetries with Poisson structures}
\author{Camille Laurent-Gengoux} \address{Camille Laurent Gengoux, Laboratoire et D\'{e}partement de Math\'{e}matiques
  UMR 7122,
  Université de Metz et CNRS, France}
\author{Eva Miranda}\address{ Eva Miranda,
Departament de Matem\`{a}tica Aplicada I, EPSEB, Universitat Polit\`{e}cnica de Catalunya, Barcelona, Spain, \it{e-mail: eva.miranda@upc.edu}}
 \thanks{Eva Miranda has been partially supported by the DGICYT/FEDER project MTM2009-07594: Estructuras Geometricas: Deformaciones, Singularidades y Geometria Integral until December 2012. Her research will be partially supported by the project GEOMETRIA ALGEBRAICA, SIMPLECTICA, ARITMETICA Y APLICACIONES with reference: MTM2012-38122-C03-01  starting in January 2013.}

\begin{document}
\maketitle

\date{\today}

\begin{abstract}
We study local normal forms for completely integrable systems on Poisson manifolds in the presence of additional symmetries. The symmetries that we consider are encoded in actions of compact Lie
groups. The existence of Weinstein's splitting theorem for the integrable system is also
studied giving some examples in which such a splitting does not exist, i.e. when the integrable system is not, locally,
a product of an integrable system on the symplectic leaf and an integrable system on a transversal. The
problem of splitting for integrable systems with additional symmetries is also considered.
\end{abstract}

\section{Introduction}

Integrable Hamiltonian systems have been widely studied in the context of symplectic manifolds.
The existence of action-angle coordinates (semilocal and global) under some additional conditions have
become a main goal in this area. For an extensive study of the existence of action-angle coordinates in symplectic manifolds,
we refer to Arnold \cite{arnold} and Duistermaat \cite{Duistermaat}.
Several extensions of the concept of complete integrability have been done in the symplectic context. We refer to Dazord and Delzant \cite{dazorddelzant}
and Nekhoroshev \cite{Nekhroshev} for more
details. However, the most natural framework for several dynamical systems is the framework of Poisson manifolds.
For instance, the Gelfand-Cetlin system, whose underlying Poisson structure is the
dual of a Lie algebra \cite{guilleminandsternberg}, arises from a non-symplectic Poisson structure.

In \cite{camillepoleva}, we proved an action-angle theorem for completely integrable systems within the Poisson context.
Indeed, we proved an action-angle theorem for non-commutative integrable systems,
non-commutative integrable systems being systems for which there are more constants of motion than required in order to prove Liouville integrability.
An important point is that a non-commutative
integrable systems may be regular even at singular points of the Poisson structure.

The construction of these \emph{action-angle} coordinates goes through the construction of a natural Hamiltonian $\mathbb T^n$ action
tangent to the fibers of the moment map. So, as it happens in the
symplectic case, we can naturally let a compact Abelian group act in the integrable system.
In \cite{camillepoleva}, however, we did not address what happens when there are additional symmetries
encoded in actions of compact Lie groups. As a first step in that direction, we prove a local equivariant
Carath\'{e}odory-Jacobi-Lie theorem \footnote{We could also call this theorem equivariant Carath\'{e}odory-Weinstein theorem
as its equivalent for symplectic forms is often called the
Darboux-Carath\'{e}odory theorem. But we prefer to stick to the old denomination Carath\'{e}odory-Jacobi-Lie
since Lie already worked partial aspects of this result in (see  Satz 3 on p. 198 in
\cite{Lie}).}   for the Poisson structure. The existence of an equivariant analogue of the action-angle
theorem for non-commutative integrable systems will be considered in a future work.

Another interesting problem in the Poisson context is that of splitting of completely integrable systems.
Weinstein's celebrated splitting theorem \cite{weinstein} asserts that locally a Poisson manifold is a
direct product of a symplectic manifold and a Poisson manifold of rank-$0$ (transversal Poisson structures).
One can attempt to apply a similar strategy taking into account
additional symmetries. Nguyen Tien Zung and the second author of this paper considered this problem in \cite{mirandazung2006}
where an equivariant splitting theorem was proved under some mild
assumptions. Still the problem of studying the possibility of splitting, not only the Poisson structure, but an integrable system and
a Poisson structure simultaneously was not considered before. We also consider the problem of splitting the integrable
system and the Poisson structure taking into account additional symmetries (equivariant splitting theorem for completely integrable systems).
We do not claim to have settled the question in the present note, but we claim to have at least clarified some elementary and less
elementary facts that shall be the basis of a more general study.

{\bf{Organization of this paper:}}

In Section 2, we obtain an equivariant Carath\'{e}odory-Jacobi-Lie theorem for Poisson structures which can be seen as a variant of Weinstein's splitting theorem \cite{weinstein} in which normal forms for a set of functions are also taken into account.

In Section 3, we consider the problem of splitting completely integrable systems. In particular, we give some counterexamples that make the question more precise.
In \cite{camilleeva}, we will present the obstruction of integrability in terms of foliated Poisson cohomology.

In Section 4, we combine results and techniques in Section 2 and Section 3 to study equivariant normal forms for regular completely integrable systems which are split.

{\bf{Acknowledgements:}}

The second author of this paper is thankful to the Hanoi National University of Education for their warm hospitality during her visit in occasion of the conference GEDYTO that she co-organized. She is
particularly thankful to Professor Do Duc Thai.
 She also wants to thank the ESF network \emph{Contact and Symplectic Topology (CAST)} for providing financial support to organize this conference.
The collaboration that led to this paper started longtime ago. The authors acknowledge financial support from the Marie Curie postdoctoral  EIF project GEASSIS FP6-MOBILITY24513.

\section{Equivariant Carath\'{e}odory-Jacobi-Lie theorem for Poisson manifolds}

We recall the definitions of a (Liouville) completely integrable system and of an involutive family on a Poisson manifold. We refer to \cite{LPV} and \cite{camillepoleva} for a more detailed introduction to that matter.

\begin{defn}
Let $(M,\Pi=\{.,.\})$ be a Poisson manifold of rank $2r$ and of dimension $n$. An $s$-tuplet of functions
  ${\F}=(f_1,\dots,f_s)$ on $M$ is said to define a \emph{(Liouville) integrable system} on $(M,\Pi)$
  if
\begin{enumerate}
    \item the functions $f_1,\dots,f_s$ are independent,
    \item for all $i,j=1, \dots,s$, the functions $f_i$ and $f_j$ are in involution, i.e. $\{f_i,f_j\}=0$,
    \item $r+s=n$.
  \end{enumerate}
When only items (1) and (2) hold true, we shall speak of a \emph{involutive family}.
\end{defn}

The map ${\F}:M\to\R^s$ is called the {\emph{momentum map}} of $(M,\Pi,\F)$. For the moment, we leave aside Liouville completely integrable systems,
and start by giving equivariant normal forms when only the first and second items in the previous definition are satisfied, but, a priori, not the third one.
The following theorem is the classical Carath\'{e}odory-Jacobi-Lie theorem, but stated within the context of Poisson manifolds, as stated in \cite{camillepoleva}.

\begin{thm} \label{thm:localsplitting}
  Let $(M,\Pi)$ be a Poisson manifold of dimension $n=2r+s$. Let $p_1,\dots,p_r$ an involutive family,
  defined on a neighborhood of some given point $m \in M$, which vanish at $m$ and whose Hamiltonian vector fields
  are linearly independent at $m$. There exist a neighborhood $U$ on $m$ and functions $q_1,\dots,q_r,z_1,\dots,z_{s}$ on $U$,
  such that
  \begin{enumerate}
    \item the $n$ functions $(p_1,q_1,\dots,p_r,q_r, z_1,\dots,z_{s})$ form a system of local coordinates, centered
    at $m$;
    \item the Poisson structure $\Pi$ is given in these coordinates by
    \begin{equation}\label{eq:thm_split}
      \Pi=\sum_{i=1}^r\pp{}{q_i}\we\pp{}{p_i}+\sum_{i,j=1}^s g_{ij}(z)\pp{}{z_i}\we\pp{}{z_j},
    \end{equation}
    where each function $g_{ij}(z)$ is a smooth function and is independent from $p_1,\dots,p_r, q_1,\dots,q_r$.
  \end{enumerate}
  The rank of $\Pi$ at $m$ is $2r$ if and only if all the matrix $(g_{ij}(z))_{i,j=1}^s$ vanishes for $z=0$.
\end{thm}

In this section we prove an equivariant version of Theorem \ref{thm:localsplitting}. Our proof follows roughly the lines of \cite{camillepoleva}.

\begin{thm}\label{thm:equivariantnormalform1}
   Let $(M,\Pi)$ be a Poisson manifold of dimension $n=2r+s$, acted upon by a compact Lie group $G$,
whose action preserves the Poisson structure. Let $p_1,\dots,p_r$ be an involutive family of $G$-invariant functions, defined on a neighborhood of  a fixed point $m$ of the $\G$-action, whose
Hamiltonian vector fields are linearly independent at $m$. There exist a neighborhood $U$ of $m$ and functions $q_1,\dots,q_r,z_1,\dots,z_{s}$ on $U$,
  such that
  \begin{enumerate}
    \item the $n$ functions $(p_1,q_1,\dots,p_r,q_r, z_1,\dots,z_{s})$
    form a system of local coordinates, centered at $m$;
    \item the Poisson structure $\Pi$ is given in these coordinates by
    \begin{equation}\label{eq:thm_split2prime}
      \Pi=\sum_{i=1}^r\pp{}{q_i}\we\pp{}{p_i}+\sum_{i,j=1}^s g_{ij}(z)\pp{}{z_i}\we\pp{}{z_j},
    \end{equation}
    where each function $g_{ij}(z)$ is a smooth function and is
    independent from $p_1,\dots,p_r, q_1,\dots,q_r$.
    \item  There exists a group homomorphism from $G$ to the group of $s \times s$
    invertible matrices, denoted by $g \to A(g)$, so that the $G$-action on
    $U$ is given, in the previous coordinates, for all $g \in G$, by
 $$ \rho \big(g  , (p_1,q_1,\dots,p_r,q_r, z_1,\dots,z_{s} ) \big) = \\
 ( p_1,q_1,\dots,p_r,q_r,\sum_{i=1}^s  m_{i,1} (g) z_i,\dots,\sum_{i=1}^s  m_{i,s} (g) z_i ) .$$
\end{enumerate}
 The rank of $\Pi$ at $m$ is $2r$ if and only if all the matrix $(g_{ij}(z))_{i,j=1}^s$ vanishes for $z=0$.
\end{thm}

\begin{proof}
 We show the theorem by induction on $r$.
 For $r=0$, the theorem reduces to the classical theorem of Bochner, that states that there exists a local system of
coordinates $z_1, \cdots, z_n $, centered at $m \in M $, on  which the action of $G$ is linear. Any such a system of coordinates satisfies
the three items of theorem \ref{thm:equivariantnormalform1}.

Now assume that the result holds true for every point in every Poisson manifold and every $r-1$ tuple of functions as above, with $r\geq1$, and let us prove that this still holds for arbitrary such
$r$-tuples. For that purpose, we consider an arbitrary point $m$ in a Poisson manifold $(M,\Pi)$, assumed to be a fixed point for the action of $G$, and we assume that we are given $G$-invariant and
Poisson commuting functions $p_1,\dots,p_r$, defined on a neighborhood of $m$, which vanish at $m$, and whose Hamiltonian vector fields are linearly independent at $m$. On a neighbourhood of $m$, the
distribution $\D:=\langle{\X}_{p_1},\dots,{\X}_{p_{r}} \rangle$ has constant rank $r$ and is an involutive distribution because $[\X_{p_i},\X_{p_j}]= -\X_{\pb{p_i,p_j}}=0$. By Frobenius Theorem, there
exist local coordinates $g_1,\dots, g_{n}$, centered at $m$, where $n:=\dim M$, such that $\X_{p_i}=\pp{}{g_i}$ for $i=1,\dots,r$,
on a neighbourhood $V $ of $m$. Setting $\tilde{q}_{r}:=g_{r}$ we
have, %
  \begin{equation} \label{eq:q1}
    {\X}_{\tilde{q}_r}[p_i]=-{\X}_{p_i}[\tilde{q}_r]= \pb{\tilde{q}_r, p_i}=-\delta_{i,r}, \qquad i=1,\dots,r.
  \end{equation}
 Without any loss of generality, one can assume that  the neighbourhood $V$ is $G$-invariant.  Let $q_r = \int_{g \in G} g^* \tilde{q}_r {\rm d}\mu  $ where $\mu $ is the normalized Haar measure on
$G$ and $g^* f $ is, for any function $f$ defined on a $g$-invariant subset of $M$,
 a shorthand for the function $m \to f(\rho(g,m)) $.
The function $ q_r$ is $G$-invariant by construction.
 Since $G$ acts by preserving the Poisson structure, and since $p_1,
\dots, p_r$ are $G$-invariant functions, we have, for all $g \in G $,
  $$ \pb{g^* \tilde{q}_r , p_i} = \pb{g^* \tilde{q}_r , g^*p_i}=  g^* \pb{\tilde{q}_r , p_i} =
-g^*\delta_{i,r} =-\delta_{i,r} \qquad i=1,\dots,r.$$ Integrating this last inequality  we obtain $\pb{ q_r , p_i} = - \delta_{i,r}$, and
 \begin{equation} \label{eq:q1prime}
    {\X}_{q_r}[p_i]=-{\X}_{p_i}[q_r]= \pb{q_r,
    p_i}= - \delta_{i,r}, \qquad i=1,\dots,r.
  \end{equation}
 In particular, the $r+1$ covectors $\diff_m p_1,\dots,\diff_m p_{r}$ and
 $\diff_m q_{r}$ of  $T^*_mM$ are linearly independent.

The distribution $\D'= \langle{\X}_{p_{r}},{\X}_{q_{r}}\rangle$ has rank $2$ at $m$, hence in a neighborhood of $m$.  It is an integrable distribution because
$[{\X}_{q_{r}},{\X}_{p_{r}}]=-{\X}_{\pb{q_{r},p_{r}}}=0.$ Applied to $\D'$, Frobenius Theorem shows that there exist local coordinates $v_1,\dots,v_{n}$, centered at $m$, such that %
\begin{equation}\label{eq:r_der}
  {\X}_{p_r}=\pp{}{v_{n-1}}\qquad\hbox{and}\qquad{\X}_{q_r}=\pp{}{v_n}.
\end{equation}%
 It is clear that the covectors $ \diff_m v_1,\dots,\diff_m v_{n-2}$ vanish on ${\X}_{p_r}(m)$ and on ${\X}_{q_r}(m)$,
 so that $(\diff_m v_1,\dots,\diff_m v_{n-2},\diff_m p_r,\diff_m q_r) $ is a
basis of $T^*_mM$. Therefore, the $n$ functions $(v_1,\dots,v_{n-2},p_{r},q_{r})$ form a system of local coordinates, centered at $m$. It follows from (\ref{eq:q1prime}) and (\ref{eq:r_der}) that the
Poisson structure takes in terms of these coordinates the following form: %
\begin{equation*}
  \Pi=\pp{}{q_r}\we\pp{}{p_r}+\sum_{i,j=1}^{n-2}h_{ij}(p_r,q_r,v_1,\dots,v_{n-2})\pp{}{v_i}\we\pp{}{v_j}.
\end{equation*}%
 The Jacobi identity, applied to the triples $(p_{r},v_i,v_j)$ and $(q_{r},v_i,v_j) $, implies that the functions $h_{ij}$ do not depend on the variables $p_{r},q_{r}$, so that %
\begin{equation} \label{eq:local1}
  \Pi=\pp{}{q_r}\we\pp{}{p_r}+\sum_{i,j=1}^{n-2}h_{ij}(v_1,\dots,v_{n-2})\pp{}{v_i}\we\pp{}{v_j},
\end{equation}
 which means that $\Pi$ is, in a $G$-invariant neighborhood of $m$, the product of a symplectic structure (on a neighborhood $V_S $ of the origin in $\R^2$) and a Poisson structure (on a neighborhood
$V_P $ of the origin in $\R^{n-2}$).

In order to apply the induction hypothesis, we need to show  that in case $r-1>0$,
 \begin{enumerate}
\item the functions $p_1,\dots,p_{r-1}$ depend only on the coordinates $v_1,\dots,v_{n-2}$, i.e., are independent of $p_r$ and $q_r$, %
\begin{equation}\label{eq:indep_r}
   \pp{p_i}{p_r} = 0=\pp{p_i}{q_r}\qquad i=1,\dots,r-1.
\end{equation}%
 \item  the action of $G$ depends only on the coordinates $v_1,\dots,v_{n-2}$, i.e. $g^* v_j $ is, for all $j=1, \dots, s$, a function that depends only on $v_1,\dots,v_{n-2}$.
\end{enumerate}

Both equalities in (\ref{eq:indep_r}) follow from the fact that $p_i$ Poisson commutes with $p_r$ and $q_r$, for $i=1,\dots, r-1$, combined with (\ref{eq:local1}): %
\begin{equation*}
  0=\pb{p_i,p_r}=\pp{p_i}{q_r},\qquad 0=\pb{p_i,q_r}=-\pp{p_i}{p_r}.
\end{equation*}%
 This proves the first point above.

For the second point, we proceed as follows: for any $g \in G$, $\rho(g, \cdot ) $ is defined in coordinates by
 $$  (p_r, q_r ,v_1,\dots,v_{n-2} ) \to \big( p_r, q_r, \alpha_1( p_r, q_r,
 v_1  ,\dots,v_{n-2} ), \dots, \alpha_{n-2}( p_r, q_r, v_1  ,\dots,v_{n-2}) \big)  $$
where we have exploited the fact that $p_r, q_r $ are $G$-invariant, and where, by definition, $\alpha_i = g^* v_i $ for $i=1, \dots n-2 $.
Now, the invariance of the Poisson bracket amounts to:
 $$ \pp{\alpha_i}{q_r}=   \pb{p_r, \alpha_i} =  \pb{p_r,
 g^* \alpha_i} =  g^* \pb{p_r,
   \alpha_i}   =  \pb{p_r , v_i} =  0 $$
  for $i=1, \dots, n-2 $.
Applying the same procedure to the functions $ p_r$ yields $ \pp{\alpha_i}{p_r} =0$.

We may now apply the induction hypothesis on the second term in (\ref{eq:local1}) and to the $r-1$-tuple of commuting functions $p_1, \dots, p_{r-1}$ to build coordinates that will satisfy the three
items of the theorem. By recursion, the theorem is valid for any integer $r$.
\end{proof}

\begin{cor}
Assume that the conditions of theorem \ref{thm:equivariantnormalform1} are satisfied, and that, moreover, the rank of $\Pi$ at $m$ is $2r$
and that the linearized part of the transversal Poisson
structure of $\Pi$ at $m$ is a semisimple compact Lie algebra $\mathfrak k$. Then the coordinates $(p_1,q_1,\dots,p_r,q_r, z_1,\dots,z_{s})$ can be chosen such that, in addition to the properties of
theorem \ref{thm:equivariantnormalform1}, we have:
$$\Pi = \sum_{i=1}^r \frac{\partial}{\partial p_i}\wedge
\frac{\partial}{\partial q_i} + \frac12\sum_{i,j,k} c^{k}_{ij} z_k \frac{\partial}{\partial z_i}\wedge \frac{\partial}{\partial z_j}$$ where $c_{ij}^k$ are structural constants of $\mathfrak k$.
\end{cor}
\begin{proof}
The corollary simply follows from a famous result by Conn \cite{conn}, stating that any smooth Poisson structure, which vanishes at a point and whose linear part at that point is of semisimple compact
type, is locally smoothly linearizable. But this result has to be adapted, since, a priori, this change of coordinates might make the action non-linear. Fortunately, we can now use a result due to
Ginzburg \cite{ginzburg} stating that if a Poisson structure $\Pi$ vanishes at a point $p$ and is smoothly linearizable near $p$, if there is an action of a compact Lie group $G$ which fixes $p$ and
preserves $\Pi$, then $\Pi$ and this action of $G$ can be simultaneously linearized.\footnote{One could envisage to implement here the ideas of Crainic and Fernandes in \cite{rui1} and \cite{rui2} that
use the idea of stability to give a different proof of  this equivariant Conn's linearization theorem. The authors did not try this approach.}
\end{proof}

\section{Split completely integrable systems}

In this section, we will assume that our {\em completely integrable system} defines a regular foliation of dimension $ r $. By an
{\em isomorphism of completely integrable system}, we mean a
diffeomorphism which preserves both the Poisson structure and the foliation defining the integrable system.

  Given two completely integrable systems, the product of both Poisson manifolds, endowed with the product
of both foliations, is again a completely integrable system, defining henceforth the {\em direct product of completely integrable system}.

\bigskip
\noindent {\bf Convention.} We  denote by ${\mathcal F}$ the sheaf of functions constant on the leaves of the foliation, which means that, for a given open subset $U \subset M$, ${\mathcal F}_U$
stands for sections over $U$. Also $T_m{\mathcal F} $ stands for the tangent space of the leaf through $m \in M$.
\bigskip
% %

\begin{remark} There are various alternative definitions of completely integrable systems : we insist on the requirement the the foliation is regular
- which can be obtained in any case by
considering only the open subset of regular points. \end{remark}

\begin{remark}
Note that not every Poisson structure can admit a completely integrable system defined as before. An obvious condition is, for instance, that at each point $m \in M$ where $\pi_m=0$,
the natural Lie algebra structure that the cotangent space $T_m^*M$ is endowed with needs to admit an Abelian Lie sub-algebra of dimension $ n-r  $, namely $
T_m{\mathcal F}^{\perp} $.
\end{remark}

  Let ${\mathcal S} $ be a symplectic leaf of a Poisson structure $(M,\pi)$. We make no assumption of regularity on ${\mathcal S} $, which means that the dimension $2r'$ of ${\mathcal S} $ may be
  strictly smaller that $2r$ - in fact the singular case is the one we are really interested in.
  Our basic assumption is the following.

\bigskip
\noindent {\bf Assumption.} From now, we assume that the restriction of the integrable system to ${\mathcal S}$ is an integrable system on a neighbourhood (in ${\mathcal S}$) of some fixed point $s
\in {\mathcal S}$, that is to say, we assume that one of the following equivalent conditions is satisfied:
\begin{enumerate}
\item there exists, in a neighbourhood $U$ of every $s \in {\mathcal
  S}$ functions $f_1,\dots, f _{r'} \in {\mathcal F}_U $ whose restrictions to ${\mathcal S} $ are independent
\item the following condition holds
$$ {\rm dim} (T_s{\mathcal F} \cap T_s {\mathcal S} ) = r' $$
for all $s \in {\mathcal S}$ (recall that  $r'=\frac{{\rm  dim} ({\mathcal S})}{2}$ by definition).
\end{enumerate}
We call {\em ${\mathcal F}$-regular} such a symplectic leaf.
\bigskip

We leave it to the reader to check the equivalence of these conditions.

Notice that every regular leaf is ${\mathcal F}$-regular. But it is  {\em not} automatically true when the leaf is singular. Here is a counter-example.

\bigskip
\noindent {\bf Counter-example.} On $M={\mathbb R}^4 $, define a Poisson structure by:
 $$ \pi = xy \pp{}{x} \we \pp{}{y} + \pp{}{q} \we \pp{}{p} .$$
Then $ {\mathcal T}=\{x=0,y=0\}$ is a singular leaf of dimension $2$. So that $n=4,r=2$ and $r'=1$.

Let $ {\mathcal F}$ be the foliation defined by the pair of functions $ \F= (x e^{p}, y e^{-q}) $. These functions are independent at all point of $ M$, hence this foliation is regular. One checks
easily the relation
 $$  \{ x e^{p}, y e^{-q}\}=0  ,$$
which guaranties that $\F$ is an integrable system. Now, $   {\rm dim} T_s ({\mathcal F}) \cap T_s {\mathcal S}  = 2 \neq r'  $.

Let us recall that   Weinstein's Splitting Theorem (\cite{weinstein}) states that,
 for every transverse submanifold  ${\mathcal T}
$ of ${\mathcal S}$ (i.e. a submanifold of $M$  crossing transversally ${\mathcal
  S} $ at a point $s \in {\mathcal S}$)
\begin{enumerate}
\item ${\mathcal T} $ is a Poisson-Dirac submanifold (in a
  neighborhood of $s$ in ${\mathcal T} $ at least) - and therefore inherits a natural Poisson structure $\pi_{\mathcal T} $.
\item the Poisson manifold $(M,\pi)$ is  (in a neighborhood of $s$
 in $M$) the direct product of the symplectic structure of ${\mathcal S} $  (restricted to a neighborhood of $s$
 in ${\mathcal S}$) with the induced Poisson structure on ${\mathcal T} $  (restricted to a neighborhood of $s$
 in ${\mathcal T}$).
\end{enumerate}

We wish to see whether or not a similar result can be established for a Poisson structure endowed with a completely integrable system. More precisely, given a ${\mathcal F} $-regular symplectic leaf
${\mathcal S} $ on a Poisson manifold $(M,\pi)$ endowed with a completely integrable system ${\mathcal F} $, and a transversal ${\mathcal T} $
through $s \in {\mathcal S} $. We want to address the
following problems.

\begin{enumerate}
 \item Does the transversal ${\mathcal T} $  inherits an integrable system (with respect to its
  induced Poisson-Dirac structure  $\pi_{\mathcal T} $) ?
\item When this is the case, is the integrable system $(M,\pi,{\mathcal F})$ the direct product of the restricted integrable system on ${\mathcal S} $ and the induced integrable system on ${\mathcal
    T}$?
\end{enumerate}

The answer to both question is negative in general, but necessary and sufficient conditions can be given to give a positive answer.
Let us start with a definition:

\begin{defn}
We say that a transversal ${\mathcal T} $ through $s \in {\mathcal S}$ is  ${\mathcal F}$-compatible when ${\mathcal F} $ restricts (at least in a neighborhood of $s$ in ${\mathcal T} $) to a regular
foliation ${\mathcal F}_{\mathcal T} $ of rank $r-r'$ on ${\mathcal T} $.
\end{defn}

Let us relate  regularity of a symplectic leaf and  existence of ${\mathcal F}$-compatible transversal,
\begin{lemma}\label{lemma:f1fr'}
Let $(M,\pi)$ be a Poisson manifold of, ${\mathcal F}$ a completely integrable system, and ${\mathcal S} $ a symplectic leaf of dimension $2r'$
\begin{enumerate}
\item The symplectic leaf $ {\mathcal S}$ is ${\mathcal F} $-regular at $s \in {\mathcal S}$ if and only if there exists a ${\mathcal F}$-compatible transversal through $s$. \item In this case,
    there exists, in a neighborhood of $s$, a foliation for which all leaves are ${\mathcal F}$-compatible transversals.
\end{enumerate}
\end{lemma}
\begin{proof}
1) If there exists a ${\mathcal F}$-compatible transversal through $s$, then the intersection of $T_s {\mathcal S}$ with $T_s{\mathcal F} $ has to have dimension $r'$, i.e. the symplectic leaf $
{\mathcal S}$ is ${\mathcal F} $-regular at $s \in {\mathcal S}$.

Conversely, let us assume that $s \in {\mathcal S}$. Let $f_1, \dots, f_{r'}$ be functions in ${\mathcal F}_U $  ($U$ a neighborhood of $s$ in $M$) whose restriction to ${\mathcal S} $ form an
integrable system on it.

By shrinking  the neighbourhood $U$ if necessary, we can extend $f_1, \dots, f_{r'}$ to a family  $f_1, \dots, f_{n-r} $ of independent functions defining the foliation ${\mathcal F}_U $. Shrinking
again $U$ if necessary, we can construct, according to the Carath\'{e}odory-Jacobi-Lie theorem \ref{thm:localsplitting} , functions $g_1, \dots, g_{r'}$ such that $\{f_i,g_j\}(s)= \delta_{i}^j $ for all
$i,j=1,\dots, r'$ for every $s \in {\mathcal S} \cap U$. We can assume  without any loss of generality that all the previously constructed
 functions vanish at $s$.

The functions $f_1,\dots, f_{r'},g_1, \dots, g_{r'} $ define local coordinates on ${\mathcal S} $, so that the following manifold is transversal to ${\mathcal S} $
 $$ {\mathcal T} = \{  f_1=\dots= f_{r'}=g_1= \dots= g_{r'}=0 \} \cap U  .$$
Shrinking $U$ again, one can assume that the functions  $f_1,\dots,f_{n-r},g_1, \dots, g_{r'} $ are independent, so that the intersection
$$ T_x {\mathcal F} \cap  T_x {\mathcal T}  $$
is the dual of the space
$$  {\rm d}_x f_{r'+1},\dots,{\rm d}_x f_{n-r}  $$
and has therefore dimension $ r-r' $. This proves the first claim.

2) The second claim is obtained by defining a foliation by:
$${\mathcal T}_{AB} = \{  f_1=a_1, \dots, f_{r'}=a_{r'}, g_1= b_1, \dots , g_{r'}=b_{r'} \} \cap U ,$$
where $A=(a_1, \dots, a_{r'})$ and $B=(b_1, \dots, b_{r'}) $ for all $A,B $ in a neighborhood of $0 \in \R^n$.
\end{proof}

The following proposition replies to the first question above

\begin{prop}
Let $(M,\pi)$ be a Poisson manifold of dimension $n$ and rank $2r$, ${\mathcal F}$ a completely integrable system, and ${\mathcal S} $ a ${\mathcal F} $-regular symplectic leaf of dimension $2r'$.
Then, for every  ${\mathcal F}$-compatible transversal ${\mathcal T} $ through $s \in {\mathcal S}$, the induced regular foliation ${\mathcal F}_{\mathcal T} $ is an integrable system with respect to
the induced Poisson structure $\pi_{\mathcal T} $.
\end{prop}
\begin{proof}
The Poisson-Dirac structure on ${\mathcal T} $ has rank $2(r-r')$, and the foliation ${\mathcal F}_{\mathcal T} $ has rank $r-r' $. The condition on dimensions is therefore satisfied, and we are left
with the task of proving that every two functions in ${\mathcal F}_{\mathcal T} $ Poisson-commute.

 This point is not trivial, for commuting functions of the ambient space do not need to commute any more,
 when restricted to a  Poisson-Dirac submanifold. Let $m \in {\mathcal
 T}$, and $f,f' \in {\mathcal F}$ be defined in a neighborhood of $m$.
 A priori $\X_f(m)$ is not tangent to ${\mathcal T} $.
 To overcome this issue, let us consider functions $f_1, \dots, f_{r'} \in {\mathcal F}$, whose restrictions
 to ${\mathcal S} $ are independent. We can moreover assume that the restriction to $ {\mathcal T}$ of these functions are equal to zero. The
 vectors $ \X_{f_1}(m), \dots, \X_{f_r}(m) $, for $m$ close
 enough from $s $, do generate a subspace in $T_m {\mathcal F} $ which
 has no intersection with $T_m {\mathcal T} $. Hence, there exists a linear
 combination of the form $ \sum_{i=1}^{r'} \lambda_i \X_{f_i}(m) $ such that
   $$  \X_f (m) -\sum_{i=1}^{r'} \lambda_i \X_{f_i}(m)   \in  T_m
   {\mathcal T}  $$
Now the restriction to ${\mathcal T} $ of $f$ and $f-\sum_{i=1}^{r'}
   \lambda_i f_i $ are equal, so that,
 by definition of the Poisson-Dirac bracket on $ {\mathcal T}  $, we have:
   $$ \{f,f'\}_{\mathcal T} (m) = \{f- \sum_i \lambda_i f_i , f'\}
   (m)= 0.$$
This completes the proof.
\end{proof}

\begin{defn}
We say that the integrable system ${\mathcal F}$ is {\rm split} at a point $s$ in a ${\mathcal F} $-regular symplectic leaf ${\mathcal S} $
if there exists a neighborhood $U $ of $s \in M $ which is
isomorphic, as a completely integrable system, to the direct product of,
\begin{enumerate}
\item the induced completely integrable system on $ {\mathcal S} \cap U $,
\item  and a completely integrable system on an open ball of dimension
    $B^{n-2r'} $, called {\rm transverse completely integrable system}.
\end{enumerate}
\end{defn}

%xxx
%Of course, the Poisson structure that the open ball of item (2) in the previous definition is endowed with is the transverse Poisson structure,
%i.e. the  induced Poisson structure on every transversal
%through $s \in {\mathcal S}$.

Said differently,  ${\mathcal F}$ is {\rm split} at a point $s$ in a ${\mathcal F} $-regular symplectic leaf ${\mathcal S} $ if and only if there exists local coordinates  $p_1,\dots, p_{r'},q_1,
\dots, q_{r'},z_1,\dots, z_s $ defined in a neighborhood $U$ of $s$ in $M$, where $p_1,\dots,p_{r'},z_1,\dots, z_{n-r-r'}$ generate ${\mathcal F}_U $, and in which $\pi$ reads:
  $$ \pi = \sum_{i=1}^{r'} \frac{\partial}{\partial q_i} \wedge \frac{\partial}{\partial p_i} +
   \sum_{0 \leq i < j \leq n -2r'} \Phi_{ij}(z)  \frac{\partial}{\partial z_i} \wedge \frac{\partial}{\partial z_j} .$$

 \begin{prop}\label{lem:localtriv}
Let $(M,\pi)$ be a Poisson manifold of dimension $n$ and rank $2r$, ${\mathcal F}$ a completely integrable system, and ${\mathcal S} $ a ${\mathcal F} $-regular symplectic leaf of dimension $2r'$.
then if the integrable system is split at ${\mathcal S} $, then all the compatible transversals have induced completely integrable systems isomorphic to the transverse integrable system.
 \end{prop}
\begin{proof}
Let $p_1,\dots, p_{r'},$
$q_1, \dots, q_{r'},z_1,\dots, z_s $ be split coordinates, where $p_1,\dots,p_{r'},$ $z_1,\dots, z_{n-r-r'}$ generate ${\mathcal F} $. The transverse manifold ${\mathcal T} $ is given by equations of the
form
 $$ p_i=\phi_i(z_1,\dots,z_s), q_i =\psi_i(z_1,\dots,z_s) \hbox{ for $i=1,\dots,r'$} .$$
 Requiring the transversal to be ${\mathcal F}$-compatible amounts to requiring that the functions
$\phi_1,\dots,\phi_{r'}$ depend only on $ z_1,\dots, z_{n-r-r'} $. Said differently, the functions $ z_1,\dots, z_{n-r-r'} $ define the induced integrable system on $T$ and on ${\mathcal T} $ as
well.

  The induced Poisson structure on ${\mathcal T}$ is precisely the
Poisson structure obtained from $\pi_T $ with the help of the gauge $B=\sum_i {\rm d}\phi_i \wedge {\rm d}\psi_i $ (which is a $2$-form) (see for instance lemma 2.2 in \cite{weinstein}) Recall that
this implies that for every $1$-form $\theta$ with $ {\rm d} \theta =B$ (form that we can assume to vanish at $z=0$), the flow $(\phi_t)$ of $\pi^{\#} \theta$, evaluated at the time $t=1$, is a
Poisson diffeomorphism between $\pi_T$ and $\pi_{\mathcal T}$. Now observe that $B$ vanishes when restricted to $ {\mathcal F}_{\mathcal T}$, i.e. the foliation defined by $ z_1,\dots, z_{n-r-r'} $,
so that $ \theta$ can be assumed to satisfy the same property by the foliated version of the Poincar\'{e} Lemma. This, in turn, amounts to the fact that $\pi^{\#} \theta$ is a vector field tangent to
the foliation $ {\mathcal F}_{\mathcal T}$. The flow $(\Phi)_t $ computed with the help of that particular $\theta $ preserves the foliation, hence it is a isomorphism of integrable system.
\end{proof}

Notice that this lemma implies that the notion of transverse completely integrable system is well-defined for split completely integrable systems.

\bigskip
\noindent {\bf Counter-example.} A completely integrable system is not, in general, split in a neighborhood of a point in a ${\mathcal F}$-regular symplectic leaf.
 Consider the following counter-example.
  On ${\mathbb R}^4$, define a Poisson structure by
  $$ \pp{}{x} \we \pp{}{y}  + p \pp{}{p}  \we \pp{}{q}  $$
 and a completely integrable system with the help of the functions $ x, p+xq$.

 The ${\mathcal F}$-compatible transversal $ x=0,y=0$ admits for
induced foliation defined by the function $p $,
 while the transversal $ x=1,y=0$ admits the induced foliation defined by the function $p+q$.
For the first one, the singular locus of the induced Poisson structure is a leaf. This is not the case for the second one. So that there exists two transversals
which admit non-isomorphic induced
completely integrable system. By proposition \ref{lem:localtriv}, this integrable system is not trivial.
\bigskip

In \cite{camilleeva}, we give a necessary and sufficient condition for local triviality of regular completely integrable systems on Poisson manifolds. This characterization is done in terms of \emph{foliated
Poisson cohomology}. Also, the conditions of a system to be split and rigid are related in \cite{evaintegrablegroup}.

\section{Equivariant split integrable systems}

To conclude this article, we present a last statement where the completely integrable systems under consideration is split and a part of it is equivariant with respect to the action of a compact Lie group $G$.

\begin{thm}\label{thm:equivariantnormalform2}
   Let $(M,\Pi)$ be a Poisson manifold of dimension $n=2r+s$, acted upon by a compact Lie group $G$,
whose action preserves the Poisson structure. Let $p_1,\dots,p_{n-r}$ be an integrable system which is globally invariant, i.e. such that the $G$-action
preserves the foliation. Given a symplectic leaf
all whose points are fixed by the $\G$-action and for which the integrable system is split, there exist, in a neighborhood $U$ of $m$,
functions $q_1,\dots,q_r,z_1,\dots,z_{s}$,
  such that
  \begin{enumerate}
    \item the $n$ functions $(p_1,q_1,\dots,p_r,q_r, z_1,\dots,z_{s})$
    form a system of local coordinates, centered at $m$;
    \item the Poisson structure $\Pi$ is given in these coordinates by
    \begin{equation}\label{eq:thm t2}
      \Pi=\sum_{i=1}^r\pp{}{q_i}\we\pp{}{p_i}+\sum_{i,j=1}^s g_{ij}(z)\pp{}{z_i}\we\pp{}{z_j},
    \end{equation}
    where each function $g_{ij}(z)$ is a smooth function and is
    independent from $p_1,\dots,p_r, q_1,\dots,q_r$.
    \item  There exists a group homomorphism from $G$ to the group of $s \times s$
    invertible matrices, denoted by $g \to A(g)$, so that the $G$-action on
    $U$ is given, in the previous coordinates, for all $g \in G$, by
 $$ \rho \big(g  , (p_1,q_1,\dots,p_r,q_r, z_1,\dots,z_{s} ) \big) = \\
 ( p_1,q_1,\dots,p_r,q_r,\sum_{i=1}^s  m_{i,1} (g) z_i,\dots,\sum_{i=1}^s  m_{i,s} (g) z_i ) .$$
\item the family $ (p_1,\dots,p_r,z_1,\dots,z_{n-2r})$ generates the integrable system. %
\end{enumerate}
\end{thm}
%xxx
\begin{proof}
The proof consists in repeating one by one the steps of the proof of theorem \ref{thm:equivariantnormalform1}
with the additional condition that all the systems of coordinates considered there must be split, i.e.
we want that the coordinates $(v_1,\dots,v_n)$ that appear in the proof of theorem \ref{thm:equivariantnormalform1}
be such that the integrable system is given by $ v_1$ and by $r-1$ of the functions $v_2, \dots,v_n $.
This can be done using a foliated version of the Frobenius Theorem that states that,
since ${\X}_{p_r}$ and ${\X}_{q_r}$ generate a foliation of dimension $2$ whose intersection with the foliation
defined by the integrable system is a foliation of rank $1$,
there exist local coordinates $v_1,\dots,v_{n}$, centered at $m$, such that
$$
  {\X}_{p_r}=\pp{}{v_{n-1}}\qquad\hbox{and}\qquad{\X}_{q_r}=\pp{}{v_n}.
$$
and such that the integrable system is generated by $v_1 $
and $r-1$ of the functions  $v_2, \dots,v_n $.
The rest of the proof follows then exactly the same lines.
\end{proof}

\bibliographystyle{amsplain}

\begin{thebibliography}{10}

\bibitem{arnold} V. I. Arnold, \emph{Mathematical methods of classical mechanics}, Springer Graduate texts in mathematics, {\bf{60}}, Second edition, Springer-Verlag, New York.

\bibitem{bo} S. Bochner, \emph{Compact groups of differentiable transformations.} Ann. of Math. (2) \textbf{46}, (1945). 372--381.

\bibitem{bolsinov}A. Bolsinov and Jovanovic, \emph{ Non-commutative integrability, moment map and geodesic flows.} Annals of Global Analysis and Geometry 23,  4, 305-322, 2003.

\bibitem{conn} J. Conn, \emph{Normal forms for smooth Poisson structures,} Ann. of Math. (2) 121 (1985), no. 3, 565--593.



\bibitem{rui1} M. Crainic and R. L. Fernandes, \emph{A geometric approach to Conn's linearization theorem},  Annals of Math 173 (2011), 1119-1137.

\bibitem{rui2} M. Crainic and R. L. Fernandes, \emph{Rigidity and flexibility in Poisson geometry}, Trav. Math., 16 (2005), 53-68.

\bibitem{dazorddelzant} P. Dazord and T. Delzant, \emph{Le problème general des variables action-angle}, Journal of Differential Geometry, {\bf{26}}, (1987), 223-251.

\bibitem{eliasson} H. Eliasson, \emph{Hamiltonian systems with Poisson Commuting integrals}, Thesis, (1984), Stokholm.

\bibitem{Duistermaat} J.J Duistermaat, \emph{On global action-angle coordinates}, Communications on pure and applied mathematics, {\bf{23}}, (1980), 687-706.


\bibitem{ginzburg} V. Ginzburg, \emph{Momentum mappings and Poisson cohomology.} Internat. J. Math. \textbf{7} (1996), no. 3, 329--358.


 \bibitem{guilleminandsternberg} V. Guillemin and S.
 Sternberg, \emph{Gel'fand-Cetlin system and quantization of the
 complex flag manifold}, Journal of functional analysis and its
 applications, {\bf{52}}, (1983), 106-128.
% % %\bibitem{grabowski} Janusz Grabowski, Giuseppe Marmo, Peter W. Michor: Construction %of completely integrable systems by Poisson mappings. Modern Physics %Letters A 14, 30 (1999) 2109-2118. %
%\bibitem{ginzburg}V. Ginzburg, \emph{Momentum mappings and Poisson cohomology.} %Internat. J. Math. \textbf{7} (1996), no. 3, 329--358. % %\bibitem{karshonginzburgguillemin}  V. Ginzburg, V.
Guillemin and  Y. Karshon,{ Moment %maps, cobordisms, and Hamiltonian group actions}, AMS, 2004.



\bibitem{camilleeva} C. Laurent-Gengoux and E. Miranda, \emph{ Splitting theorem and integrable systems in Poisson manifolds}, in preparation, 2012.


\bibitem{camillepoleva} C. Laurent-Gengoux, E. Miranda and P. Vanhaecke, \emph{Action-angle coordinates for integrable systems on Poisson manifolds,}  Int. Math. Res. Not. IMRN  2011,  no. 8,
    1839--1869.

\bibitem{LPV} C. Laurent-Gengoux, A. Pichereau and P. Vanhaecke, \emph{Poisson structures,}  Grundlehren der mathematischen Wissenschaften, {\bf 347}, 2012.

\bibitem{libermannmarle} P. Libermann and C.-M. Marle, \emph{Symplectic Geometry and Analytical Mechanics}, Mathematics and its applications, Reidel Publishing company, 1987.


\bibitem{Lie} S. Lie and F. Engel, \emph{ Theorie der Transformationsgruppen}, vol. 2,
 reprint of the 1890 edition, Teubner Verlag, Leipzig, 1930.
\bibitem{evaintegrablegroup} E. Miranda, \emph{Integrable systems and group actions}, submitted, 2012.

%\bibitem{evathesis} E. Miranda, \emph{ On symplectic linearization of singular Lagrangian %foliations},  Ph. D. thesis, Universitat de Barcelona,   2003, ISBN: %9788469412374.



\bibitem{mirandazungequiv} E. Miranda and Nguyen Tien Zung,  \emph{Equivariant normal forms for nondegenerate singular orbits of integrable Hamiltonian systems},Ann. Sci. Ecole Norm. Sup.,37 (2004),
    no. 6, 819--839.


 \bibitem{mirandazung2006} E. Miranda and N. T. Zung, \emph{
A note on equivariant normal forms of  Poisson structures}, Math. Research Notes, 2006, vol 13-6, 1001--1012.



\bibitem{mirandatenerife} E. Miranda, \emph{Some rigidity results for Symplectic and Poisson group actions}, XV International Workshop on Geometry and Physics,  Publ. R. Soc. Mat. Esp.,
    R. Soc. Mat. Esp., Madrid, 2007,11, 177--183,.

\bibitem{mimozu} E. Miranda, P. Monnier and N.T.Zung, \emph{ Rigidity of Hamiltonian actions on Poisson manifolds},  Adv. Math. 229 (2012), no. 2, 1136-1179.

 \bibitem{Nekhroshev} N. N. Nekhroshev, \emph{Action-angle
 variables and their generalizations}, Translations of Moscow
 Mathematical Society, {\bf{26}}, (1972), 181-198.
% %\bibitem{sadetov} S.T. Sadètov, \emph{ A proof of the Mishchenko-Fomenko conjecture (1981)}. %Dokl. Akad. Nauk 397 (2004), no. 6, 751--754.



\bibitem{weinstein} A. Weinstein, \emph{The local structure of Poisson manifolds.}, J. Differential Geom. 18 (1983), no. 3, 523--557.


 \bibitem{tienzung} N. T. Zung, \emph{Symplectic topology of
 integrable Hamiltonian systems. I. Arnold-Liouville with
 singularities.} Compositio Math., 101(2),(1996) 179-215.

 \bibitem{voro} Y. Vorobjev, \emph{Coupling tensors and Poisson geometry near a single %%@
symplectic leaf.} Lie algebroids and related topics in differential geometry (Warsaw, 2000), 249--274, Banach Center Publ., 54, Polish Acad. Sci., Warsaw, 2001.

}
 \end{thebibliography}

\providecommand{\bysame}{\leavevmode\hbox to3em{\hrulefill}\thinspace}

\end{document}